\documentclass[12pt]{article}
\usepackage {amsmath, amssymb, amsthm}
\def\ep{\varepsilon}
\def\th{\vartheta}

\newtheorem{tm}{Theorem}[section]
\newtheorem{prop}{Proposition}[section]
{\theoremstyle{definition}
\newtheorem{rem}{Remark}[section]}

\begin{document}
{\large
\centerline{\normalsize \bf SMALL BALL PROBABILITIES FOR SMOOTH GAUSSIAN FIELDS} 
\centerline{\normalsize \bf  AND TENSOR PRODUCTS OF COMPACT OPERATORS }  
\bigskip 

\centerline {\small \rm BY ANDREI I. KAROL' AND ALEXANDER I. NAZAROV\footnote{Authors are 
partially supported by grants of RFBR No. 10-01-00154a and NSh. 4210.2010.1 
}}

\bigskip

\centerline {\it St.Petersburg State University}

}

\bigskip

{ \it \footnotesize We find the logarithmic $L_2$-small ball
asymptotics for a class of zero mean Gaussian fields with
covariances having the structure of ``tensor product''. The main
condition imposed on marginal covariances is slow growth at the origin of 
counting functions of their eigenvalues. That is valid for
Gaussian functions with smooth covariances. Another type of marginal 
functions considered as well are classical Wiener process, Brownian
bridge, Ornstein--Uhlenbeck process, etc., in the case of special
self-similar measure of integration. Our results are based on a new 
theorem on spectral asymptotics for the tensor products of compact 
self-adjoint operators in Hilbert space which is of independent interest.
Thus, we continue to develop the approach proposed in the paper \cite{KNN}, 
where the regular behavior at infinity of marginal eigenvalues was assumed.
}

\bigskip


\bigskip

\section{Introduction} 

The theory of small deviations of random
functions is currently in intensive development. In this paper we
address small deviations of Gaussian random fields in $L_2$-norm.

Suppose we have a real-valued Gaussian random field $X(x)$, $x\in {\cal O}\subset
{\mathbb R}^d$, with zero mean and covariance function
$G_X(x,y)=EX(x)X(y)$ for $x,y\in{\cal O}$. Let $\mu$ be a finite measure 
on ${\cal O}$. Set
\[
\|X\|_{\mu}=\|X\|_{L_2({\cal O};\mu)}=(\int\limits_{\cal O} X^2(x)\ \mu(dx))^{1/2}
\]
(the subscript $\mu$ will be omitted when $\mu$ is the Lebesgue measure) 
and consider
\[
Q(X,\mu \ ;\ep)={\bf P}\{\|X\|_{\mu}\leq \ep \}.
\]
The problem is to evaluate the behavior of $Q(X,\mu\ ;\ep) $ as 
$\ep \rightarrow 0$. Note that the case 
$\mu(dx)=\psi(x)dx$, $\psi\in L_1({\cal O})$, can be easily reduced to the 
Lebesgue case $\rho\equiv1$ replacing $X$ by the Gaussian field $X \sqrt{\rho}$. 
In general case, by scaling, one can assume that $\mu({\cal O})=1$.

According to the well-known Karhunen-Lo\`eve expansion, we
have 
$$\|X\|_{\mu}^2 \overset{d}{=} \sum_{n=1}^\infty \lambda_n\xi_n^2,$$
where $\xi_n$, $n\in\mathbb N$, are independent standard
normal r.v.'s, and $\lambda_n>0$, $n\in\mathbb N$,
$\sum\limits_n\lambda_n <\infty $, are the eigenvalues of the
integral equation 
$$ \lambda f(x)=\int\limits_{\cal O}G_X(x,y)f(y)\mu(dy), \qquad x\in {\cal O}.
$$ 
Thus we arrive to the equivalent problem of studying the asymptotic behavior 
of ${\bf P}\left\{\sum_{n=1}^\infty \lambda _n \xi_n^2\leq\ep ^2 \right\}$
as $\ep \to 0+ $. The answer heavily depends on the available information about the
eigenvalues $\lambda_n$.

The study of small deviation problem was initiated by Sytaya \cite{S}
and continued by a number of authors. See the history of the problem 
and the summary of main results in two excellent reviews 
\cite{Lf} and \cite{LiS}. The references to the latest results on 
$L_2$-small deviation asymptotics can be found on the site 
\cite{site}.\medskip

In this paper we continue to study the small deviation asymptotics of 
a vast and important class of Gaussian random fields having the
covariance of ``tensor product'' type. It means that this covariance
can be decomposed in a product of ``marginal'' covariances depending
on different arguments. The classical examples of such fields are
the Brownian sheet and the Brownian pillow. The notion of tensor 
products of Gaussian processes or Gaussian measures is known long 
ago. Such Gaussian fields are also studied in related domains, see, 
e.g., \cite{PW} and \cite{LP}. We recall briefly the construction
of these fields.

Suppose we have two Gaussian random functions $X(x)$, $x \in {\mathbb R}^m$, 
and $Y(y)$, $y \in {\mathbb R}^n$, with zero means
and covariances $G_X(x,u)$, $x,u \in {\mathbb R}^m$, and
$G_Y(y,v)$, $y,v \in {\mathbb R}^n$, respectively. Consider
the new Gaussian function $Z(x,y)$, $x \in {\mathbb R}^m$, $y\in {\mathbb R}^n$, 
which has zero mean and the covariance
$$G_Z((x,y),(u,v)) =G_X(x,u)G_Y(y,v).$$
Such Gaussian function obviously exists, and the integral operator
with the kernel $G_Z$ is the tensor product of two ``marginal''
integral operators with the kernels $G_X$ and $G_Y$. 
Therefore we use in the sequel the notation $Z = X\otimes Y$ and 
we call the Gaussian field $Z$ {\it the tensor product} of the 
fields $X$ and $Y$. The generalization to the multivariate case 
when obtaining the fields $\underset{j=1}{\overset{d}\otimes} X_j$ 
is straightforward.\medskip

The investigations of small deviations of Gaussian random functions of 
this class were started in a classical paper \cite{Cs} where the logarithmic
$L_2$-small ball asymptotics was obtained for the Brownian sheet 
${\mathbb W}_d={\mathbb W}_d(x_1,...,x_d)$ on the unit cube. 
This result was later extended by Li \cite{Li} to some other random fields. 

In a very interesting paper \cite{FT} the {\it exact}
asymptotics of small deviations in $L_2$-norm for the Brownian
sheet was obtained using the Mellin transform. However, it is not clear 
if the method of \cite{FT} yields small deviation results for a more general 
class of Gaussian fields.

A new approach developed in the paper \cite{KNN} is based on
abstract theorems describing the spectral asymptotics of tensor
products and of sums of tensor products for self-adjoint operators
in Hilbert space. This approach gave the opportunity to consider
quite general class of tensor products with the eigenvalues 
$\lambda_n^{(j)}$ of marginal covariances having the so-called 
{\it regular behaviour}: 
$$\lambda_n^{(j)}\sim\frac {\varphi_j (n)} {n^{p_j}},\, n \to \infty,$$
where $p_j>1$ and $\varphi_j$ are some {\it slowly varying functions} (SVFs).\medskip

In this paper we consider the case where $\lambda_n^{(j)}$ have faster rate
of decreasing. To be more precise, we assume that the so-called {\it counting functions}
$${\cal N}_j(t)=\#\{n:\ \lambda_n^{(j)}>t\}$$
are SVFs. Such behavior of eigenvalues is typical for processes with smooth covariances,
see, e.g., \cite{Na1}.\medskip

The structure of the paper is as follows. In \S 2 we
present some auxiliary information about slowly varying functions.
Next, in \S 3 we prove new results on the spectral asymptotics
for tensor products of compact self-adjoint operators with slowly varying
counting functions. Then, in \S 4, using the result of \cite{Na1}, we
evaluate the small ball constants for the special rate of eigenvalues decay, 
namely,
$${\cal N}(t) \sim \ln^p(1/t)\Phi(\ln(1/t)) \,,\qquad t\to 0+,$$
with $p\ge0$ and $\Phi$ being a SVF. Finally, we apply our theory to various 
specific examples of Gaussian random fields.

\section{ Auxiliary information on SVFs} 
We recall that a positive
function $\varphi(\tau)$, $\tau>0$, is called a {\it slowly
varying} function (SVF) at infinity if for any $c>0$
\begin{equation}\label{slow}
\varphi(c\tau)/\varphi(\tau)\to 1 \qquad \mbox{as}\quad
\tau\to+\infty.
\end{equation}

It is easily seen that any smooth positive function $\varphi$
satisfying $\tau\varphi'(\tau)/\varphi(\tau)\to 0$ as
$\tau\to+\infty$ is slowly varying. This test shows that the
functions equal to $\ln^p(\tau)\ln^{\varkappa}(\ln(\tau))$ for $\tau>>1$
are slowly varying.\medskip

We need some simple properties of SVFs. Their proofs can be found,
for example, in \cite{Se}.

\begin{prop}\label{SVF}
Let $\varphi$ be a SVF. Then following statements are true:

{\bf 1}. The relation (\ref{slow}) is uniform with respect to $c\in[a,b]$
for $0<a<b<+\infty$.

{\bf 2}. There exists an equivalent SVF\ $\widehat\varphi\in {\mathcal
C}^2({\mathbb R}_+)$ (i.e. $\frac {\widehat\varphi(\tau)}{\varphi(\tau)}\to1$ as
$\tau\to+\infty$) such that
\begin{equation}\label{equiv_slow}
\tau\cdot(\ln(\widehat\varphi))'(\tau)\to 0,\qquad
\tau^2\cdot(\ln(\widehat\varphi))''(\tau)\to 0,\qquad\tau\to+\infty.
\end{equation}

{\bf 3}. The function $\tau\mapsto\tau^p\varphi(\tau)$, $p>0$, up to equivalence
at infinity, is monotone for large $\tau$,
and its inverse function is $\tau\mapsto\tau^{1/p}\phi(\tau)$, where $\phi$ is a
SVF.
\end{prop}

For two nondecreasing and unbounded SVFs $\varphi$ and $\psi$, we define their {\it asymptotic convolution}
$$(\varphi\star\psi)(\tau)=\int\limits_1^{\tau} \varphi( \tau / \sigma)\,d\psi(\sigma).
$$

\begin{rem}
It is easy to see that the asymptotic convolution is connected with the Mellin convolution (see 
\cite[\S 2]{KNN}) by the relation 
$$(\varphi\star\psi)(\tau)=(\varphi*\psi_1)(\tau),\qquad \mbox{where}\qquad \psi_1(\tau)=\tau\psi'(\tau).$$
Therefore, basic properties of the asymptotic convolution could be extracted from \cite[Theorem 2.2]{KNN}. 
However, for the reader's convenience, we give them with full proofs.
\end{rem}

\begin{tm}\label{convolution}
The following statements are true:

{\bf 1}. $(\varphi\star\psi)(\tau)\le \varphi(\tau)\psi(\tau)$.

{\bf 2}. $\varphi(\tau) = o((\varphi\star\psi)(\tau))$ as $\tau\to+\infty$.

{\bf 3}. Asymptotic symmetry:
$$(\varphi\star\psi)(\tau) = (\psi\star\varphi)(\tau) + O(\varphi(\tau)+\psi(\tau)),\qquad \tau\to+\infty.
$$

{\bf 4}. If $\varphi(\tau)=\widehat\varphi(\tau)\cdot(1+o(1))$ as $\tau\to+\infty$, then
$$(\varphi\star\psi)(\tau)=(\widehat\varphi\star\psi)(\tau)\cdot(1+o(1)),\qquad \tau\to+\infty.$$

{\bf 5}. $\varphi\star\psi$ is a nondecreasing SVF.
\end{tm}

\begin{rem}
Note that, by the statement {\bf 3}, the statements {\bf 2} and {\bf 4} hold true with replacing $\varphi$ by $\psi$.
\end{rem}

\begin{proof}
{\bf 1}. This fact is trivial, since $\varphi( \tau / \sigma)\le\varphi(\tau)$ for all 
$\sigma\in\,[1,\tau]$.\medskip

{\bf 2}. By Proposition \ref{SVF}, part~{\bf 1}, for any $a>1$ we have
$$(\varphi\star\psi)(\tau)>\int\limits_1^a \varphi(\tau /\sigma)\,d\psi(\sigma)= 
\varphi(\tau)(\psi(a)-\psi(1))\cdot(1+o(1)),\qquad \tau\to+\infty.
$$
Since $\psi$ is unbounded, the statement follows.\medskip

{\bf 3}. Integrating by parts and changing the variable, we obtain
$$(\varphi\star\psi)(\tau)=\varphi(1) \psi(\tau) - \varphi(\tau)\psi(1) + (\psi\star\varphi)(\tau).
$$
By {\bf 2}, the statement follows.\medskip

{\bf 4}. By {\bf 2} and {\bf 3}, for any $a>1$
$$(\varphi\star\psi)(\tau)\sim(\psi\star\varphi)(\tau)\sim
\int\limits_a^{\tau}\psi(\tau/\sigma)\,d\varphi(\sigma)\sim
\int\limits_1^{\frac {\tau}a}\varphi(\tau/\sigma)\,d\psi(\sigma),
\qquad\tau\to+\infty.$$
Given $\ep>0$, one can take $a$ so large that 
$1-\ep<\frac {\widehat\varphi(\lambda)}{\varphi(\lambda)}<1+\ep$ for $\lambda>a$, and the statement follows.\medskip

{\bf 5}. Due to {\bf 4} and to Proposition \ref{SVF}, part~{\bf 2}, we can assume $\varphi$ and $\psi$ smooth. 
We have
 $$\tau\cdot(\varphi\star\psi)'(\tau) = \tau\psi '(\tau) \varphi(1) +
 \int\limits_1^\tau\frac{\tau}{\sigma}\cdot\varphi '(\tau/\sigma)\,d\psi(\sigma).
 $$
By {\bf 2}, 
$$\tau\psi'(\tau)=o(\psi(\tau))=o((\varphi\star\psi)(\tau)),\qquad \tau\to+\infty.$$
 Next, due to (\ref{equiv_slow}) we have for any $a>1$ and $\tau>a$
\begin{multline*}
\left|\int\limits_1^\tau\frac{\tau}{\sigma}\cdot\varphi '(\tau/\sigma)\,d\psi(\sigma)\right|\le
\left|\int\limits_1^{\frac {\tau}a}\frac{\tau}{\sigma}\cdot\varphi '(\tau/\sigma)\,d\psi(\sigma)\right|+
\left|\int\limits_{\frac {\tau}a}^\tau\frac{\tau}{\sigma}\cdot\varphi '(\tau/\sigma)\,d\psi(\sigma)\right|\le\\
\le \sup\limits_{\lambda\ge a}\left|\frac {\lambda\varphi'(\lambda)}{\varphi(\lambda)}\right|\cdot (\varphi\star\psi)(\tau)+\sup\limits_{\lambda\le a}|\lambda\varphi'(\lambda)|\cdot\psi(\tau).
\end{multline*}
By Proposition \ref{SVF}, part~{\bf 2}, given $\ep>0$, one can take $a$ so large that 
$\left|\frac {\lambda\varphi'(\lambda)}{\varphi(\lambda)}\right|<\ep$ for $\lambda>a$. 
This gives, subject to {\bf 2},
$\tau\cdot(\varphi\star\psi)'(\tau)=o((\varphi\star\psi)(\tau))$ as $\tau\to+\infty$, and the statement follows.
\end{proof}

{\bf Example 1}. Let 
$$\varphi(\tau)=\ln^\alpha(\tau)\cdot\Phi(\ln(\tau)),\qquad
\psi(\tau)=\ln^\beta(\tau)\cdot\Psi(\ln(\tau)),
$$ 
where $\alpha,\ \beta\ge0$ while $\Phi$ and $\Psi$ are SVFs\footnote{If $\alpha=0$ (respectively, $\beta=0$), then we require
in addition that $\Phi$ (respectively, $\Psi$) is nondecreasing and unbounded.}.

Then, as $\tau\to+\infty$,
\begin{equation}\label{conv_pq}
(\varphi\star\psi)(\tau)\sim\frac {\Gamma(\alpha+1)\Gamma(\beta+1)}{\Gamma(\alpha+\beta+1)}\cdot\varphi(\tau)\psi(\tau).
\end{equation}

\begin{proof}
Changing the variable we obtain
$$(\varphi\star\psi)(\tau)=
\int\limits_0^{\ln(\tau)}(\ln(\tau)-s)^\alpha\,\Phi(\ln(\tau)-s)\,d\big[s^\beta\,\Psi(s)\big].
$$

First, let at least one of exponents $\alpha$ and $\beta$ be positive. By Theorem \ref{convolution}, part~{\bf 3}, one can assume that $\beta>0$. Then, substituting $s=\ln(\tau)\th$ we get
\begin{multline*}
(\varphi\star\psi)(\tau)=\varphi(\tau)\psi(\tau)\times\\
\times\int\limits_0^1(1-\th)^\alpha\,\frac {\Phi(\ln(\tau)(1-\th))}{\Phi(\ln(\tau))}\cdot
\th^{\beta-1}\,\frac {\beta\Psi(\ln(\tau)\th)+\ln(\tau)\th\Psi'(\ln(\tau)\th)}{\Psi(\ln(\tau))}\,d\th.
\end{multline*}

By Proposition \ref{SVF}, part {\bf 3}, for any $\ep>0$ the function $T^{\ep}\Phi(T)$
increases for large $s$. Hence for $0<z<1$, $T>0$ we have 
$$\frac {\Phi(zT)}{\Phi(T)}=\frac 1{z^\ep}\cdot\frac {(zT)^{\ep}\Phi(zT)}{T^{\ep}\Phi(T)}\le 
\frac 1{z^\ep}.
$$
This estimate and a similar estimate for $\Psi$ give us the majorant required in Lebesgue
Dominated Convergence Theorem. Passing to the limit in the integral we obtain as
$\tau\to+\infty$
$$(\varphi\star\psi)(\tau)\sim\varphi(\tau)\psi(\tau)\cdot\int\limits_0^1(1-\th)^\alpha\cdot
\beta\th^{\beta-1}\,d\th,
$$
and we arrive at (\ref{conv_pq}).\medskip

Now let $\alpha=\beta=0$. Then for any $\delta\in\,]0,1[$
$$(\varphi\star\psi)(\tau)\ge\int\limits_0^{\delta\ln(\tau)}\Phi(\ln(\tau)-s)\,d\Psi(s)\ge
\Phi(\ln(\tau)(1-\delta))\cdot[\Psi(\ln(\tau)\delta)-\Psi(0)].
$$
By definition of SVF, we obtain
$$\liminf\limits_{\tau \to +\infty}\frac{(\varphi\star\psi)(\tau)}{\varphi(\tau)\psi(\tau)}\ge 1.
$$
Taking into account Theorem \ref{convolution}, part {\bf 1}, we arrive at (\ref{conv_pq}).
\end{proof}

{\bf Example 2}. Let 
$$\varphi(\tau)= \exp(a^{\frac 1{p'}}\ln^{\frac 1p}(\tau))\,\ln^{\alpha}(\tau)\cdot\Phi(\ln (\tau)),\qquad 
\psi(\tau)= \exp(b^{\frac 1{p'}}\ln^{\frac 1p}(\tau))\,\ln^{\beta}(\tau)\cdot\Psi(\ln (\tau)),
$$
where $a,b>0$, $\alpha,\beta\ge0$, $p>1$, $p'$ stands for the H\"older conjugate exponent,
while $\Phi$ and $\Psi$ are SVFs.

Then, as $\tau\to+\infty$,
\begin{equation}\label{conv_exp}
(\varphi\star\psi)(\tau)\sim\sqrt\frac{2\pi}{p-1}\ \frac
{a^{\alpha+\frac 12}b^{\beta+\frac 12}}{(a+b)^{\gamma +\frac 12}}\
\exp((a+b)^{\frac 1{p'}}\ln^{\frac 1p}(\tau))\,\ln^{\gamma}(\tau)\cdot \Phi(\ln (\tau))\Psi(\ln (\tau)),
\end{equation}
where $\gamma=\alpha+\beta+\frac 1{2p}$.\medskip

\begin{proof}
Similarly to Example 1, we have
\begin{multline*}
(\varphi\star\psi)(\tau) \sim {\frac {b^{\frac 1{p'}}}p}\,T^{\alpha+\beta+\frac 1p}\cdot\Phi(T)\Psi(T)\times\\
\times\int\limits_0^1 \exp( T^{\frac 1p}S(\vartheta))\cdot
(1-\vartheta)^\alpha\vartheta^{\beta-\frac 1{p'}}\cdot\frac {\Phi(T(1-\vartheta))}{\Phi(T)}\cdot 
\frac {\Psi(T\vartheta)}{\Psi(T)} \,d\vartheta,
\end{multline*}
where $T=\ln(\tau)$, 
$S(\vartheta) =  a^{\frac 1{p'}}(1-\vartheta)^{\frac 1p} + b^{\frac 1{p'}}\vartheta^{\frac 1p}$.

Denote by $\vartheta_\star$ the maximum point of $S(\vartheta)$. 
Then, using the Laplace method and Proposition \ref{SVF}, part~{\bf 1}, we have
$$(\varphi\star\psi)(\tau)\sim
{\frac {b^{\frac 1{p'}}}p}\,T^{\alpha+\beta+\frac 1{2p}}\cdot\Phi(T)\Psi(T)
\sqrt\frac{2\pi}{-S''(\vartheta_\star)}
\exp(T^{\frac 1p}S(\vartheta_\star))
(1-\vartheta_\star)^\alpha\vartheta_\star^{\beta-\frac 1{p'}},
$$

Direct calculation shows that
$$\vartheta^\star =\frac b{a+b}, \qquad S(\vartheta_\star) = (a+b)^{\frac 1{p'}},\qquad 
S^{''}(\vartheta^\star) = -\frac 1{pp'} \frac {(a+ b)^{3-{\frac 1p}}}{ab},$$ 
and we arrive at (\ref{conv_exp}).
\end{proof}

\section{Spectral asymptotics for tensor products of compact
self-adjoint operators}

We recall that for a compact self-adjoint nonnegative operator 
 ${\cal T}={\cal T}^*\ge0$ in a Hilbert space $H$ we denote by 
$\lambda_n =\lambda_n({\cal T})$ its positive eigenvalues
arranged in a nondecreasing sequence and repeated according to
their multiplicity. Also we introduce the counting function 
$${\cal N}(t)={\cal N}(t,{\cal T})=\#\{n: \lambda_n ({\cal T})>t\}.
$$ 

Note that in view of Proposition \ref{SVF}, part~{\bf 2}, any SVF 
arising in an asymptotic formula can be assumed smooth. 

\begin{tm}\label{ocenka} 
Let ${\cal T}$ and $\widetilde {\cal T}$ be compact self-adjoint nonnegative 
operators in Hilbert spaces $H$ and $\widetilde H$, respectively. Let
\begin{equation}\label{asymp}
{\cal N}(t)\sim \varphi(1/t),\qquad \widetilde {\cal N}(t)\equiv {\cal N}(t,\widetilde {\cal T})
\sim\widetilde\varphi(1/t),\qquad t\to+0,
\end{equation}
where $\varphi$ and $\widetilde\varphi$ are unbounded SVFs at infinity.

Then for any $\ep>0$ the operator ${\cal T}\otimes\widetilde {\cal T}$
in the space $H\otimes\widetilde H$ satisfies the following estimates:
\begin{equation}\label{ocenka_oplus}
{\cal N}_{\otimes}(t)\equiv{\cal N}(t,{\cal T}\otimes\widetilde{\cal T}) 
\lessgtr \alpha_{\pm}(\ep)
\cdot\left[ S(t,\ep)+\widetilde S(t,\ep)+
\int\limits_{\alpha_{\mp}(\ep)/\ep}^{\ep\tau}
\varphi(\tau/\sigma)\,d\widetilde\varphi(\sigma)\right] 
\end{equation}
uniformly with respect to $t>0$ (here $\tau$ stands for
$\alpha_{\pm}(\ep)/t$). For $\alpha_{\mp}(\ep)/\ep >\ep\tau$
the integral in (\ref{ocenka_oplus}) should be omitted.

In (\ref{ocenka_oplus}) $\alpha_{\pm}(\ep)\to1$ as $\ep\to+0$, and the functions
$S$, $\widetilde S$ satisfy the following relations as $t\to+0$:
\begin{equation}\label{S}
S(t,\ep)\sim \varphi(1/t)\widetilde\varphi(1/\ep);\qquad
\widetilde S(t,\ep)=o(\widetilde\varphi(1/t)). 
\end{equation}
\end{tm}

\begin{proof}
 We establish the upper bound for ${\cal N}_{\otimes}(t)$.
The lower estimate can be proved in the same way. We have
$${\cal N}_{\otimes}(t) =\sum\limits_n {\cal N}(t/\widetilde\lambda_n)=
\sum\limits_{\widetilde\lambda_n<\ep}{\cal N}(t/\widetilde\lambda_n)+S(t,\ep),$$

\noindent where
$$S(t,\ep)=\sum\limits_{\widetilde\lambda_n\ge\ep}
{\cal N}(t/\widetilde\lambda_n).$$ 
The asymptotics (\ref{S}) for the last sum follows from Proposition \ref{SVF}, part~{\bf 1}. 

Denote by $\widetilde\theta$ the inverse function for $\widetilde\varphi$. 
Then the second relation in (\ref{asymp}) implies $\widetilde\lambda_n/\widetilde\theta(n)\to1$ as $n\to\infty$, 
and we have 
$$\alpha_-(\ep)\widetilde\theta(n)\le\widetilde\lambda_n\le
\alpha_+(\ep)\widetilde\theta(n)$$ for $\widetilde\lambda_n<\ep$,
with $\alpha_{\pm}(\ep)\to1$ as $\ep\to+0$.

Using monotonicity of $\cal N$ we get
$$\sum\limits_{\widetilde\lambda_n<\ep} {\cal N}(t/\widetilde\lambda_n)\le
\sum\limits_{\widetilde\theta(n)<\ep\alpha_-^{-1}(\ep)} {\cal N}(t/\alpha_+(\ep)\widetilde\theta(n)),
$$ 
and by monotonicity of the function $n\mapsto {\cal N}(t/\alpha_+(\ep)\widetilde\theta(n))$ we estimate 
the sum by an integral:
\begin{equation}\label{q}
\sum\limits_{\widetilde\lambda_n<\ep} {\cal N}(t/\widetilde\lambda_n)\le 
{\cal N}\left(\frac {\alpha_-(\ep)t} {\alpha_+(\ep)\ep}\right)+
\int\limits_0^{\ep \alpha^{-1}_-(\ep)}{\cal N}(t/\alpha_+(\ep)\theta) (-d\widetilde\varphi(1/\theta)).
\end{equation}

The first term in (\ref{q}) is $O(\varphi(1/t))$. Therefore,
adding it to the term $S(t,\ep)$ we obtain the term
$\alpha_+(\ep)S(t,\ep)$ with $\alpha_+(\ep)\to1$ as $\ep\to 0+$.
Further, 
splitting the integral in (\ref{q}) and changing variables we obtain
\begin{equation}\label{qq}
\int\limits_{\ep}^{+\infty}{\cal N}(s)\,d\widetilde\varphi(\alpha_+(\ep)s/t)+
\int\limits_{\alpha_-(\ep)/\ep} ^{\ep\tau}{\cal N}(\sigma/\tau)\,d\widetilde\varphi(\sigma).
\end{equation}

The first integral in (\ref{qq}) gives us the term $\widetilde S(t,\ep)$. Integrating by parts 
(note that for $s>\|T\|$ the integrand equals 0) we obtain
$$\widetilde S(t,\ep)={\cal N}(\ep)\widetilde\varphi(\alpha_+(\ep)\ep/t)
-\int\limits_{\ep}^{\|T\|}\widetilde\varphi(\alpha_+(\ep)s/t)\,d{\cal N}(s).
$$
By Proposition \ref{SVF}, part~{\bf 1}, 
$\frac {\widetilde\varphi(\alpha_+(\ep)s/t)}{\widetilde\varphi(1/t)}\to1$ as $t\to 0+$ 
uniformly with respect to $s\in[\ep,\|T\|]$ , and we arrive at (\ref{S}).

By the first relation in (\ref{asymp}) we can estimate $\cal N(\sigma/\tau)$ in the second integral by
$\alpha_+(\ep)\varphi(\tau/\sigma)$ that gives the integral term in (\ref{ocenka_oplus}).
\end{proof}

\begin{tm}\label{spectr_asymp}
Let operators $\cal T$ and $\widetilde{\cal T}$ be as in Theorem $\ref{ocenka}$. Then
\begin{equation}\label{asymp_oplus}
{\cal N}_{\otimes}(t) \sim \phi(1/t)\equiv(\varphi\star\widetilde {\varphi})(1/t).
\end{equation}
\end{tm}

\begin{proof}
Fix arbitrary $\ep >0$ and consider the estimates (\ref{ocenka_oplus}). By Theorem \ref{convolution}, part~{\bf 2}, 
we have 
$$S(t,\ep) = o((\varphi\star\widetilde{\varphi})(1/t)),\quad 
\widetilde{S}(t,\ep)=o((\varphi\star\widetilde{\varphi})(1/t)), \qquad t\to+0.$$ 
Further, the integral in the right-hand side of (\ref{ocenka_oplus}) differs from the
convolution $(\varphi\star\widetilde \varphi)(\tau)$ by the integrals
$$\int\limits_{\ep \tau}^{\tau}\varphi(\tau/\sigma)\,d\widetilde{\varphi}(\sigma)=
O(\widetilde{\varphi}(\tau))=o((\varphi\star\widetilde{\varphi})(\tau)),\qquad \tau \to +\infty,$$
$$\int\limits_1^{\alpha_{\mp}/\ep}\varphi(\tau/\sigma)\,d\widetilde{\varphi}(\sigma)= 
O(\varphi(\tau))= o((\varphi\star\widetilde{\varphi})(\tau)), \qquad\tau \to +\infty,$$
(we recall that $\tau=\alpha_{\pm}(\ep)/t$).

Due to Theorem \ref{convolution}, part~{\bf 5},
$(\varphi\star\widetilde{\varphi})(\tau) \sim(\varphi\star\widetilde{\varphi})(1/t)$, and hence
$$\limsup\limits_{t\to 0+}\frac{{\cal N}_{\otimes}(t)}{\phi(1/t)}\le \alpha_+(\ep),\qquad 
\liminf\limits_{t\to 0+}\frac{{\cal N}_{\otimes}(t)}{\phi(1/t)}\ge \alpha_-(\ep),
$$ 
where $\phi$ is defined in (\ref{asymp_oplus}). Taking into account
that $\alpha_{\pm}(\ep)\to 1$ as $\ep \to 0+$, we arrive at
(\ref{asymp_oplus}).
\end{proof}

\section{Small ball asymptotics. Examples}

To connect the given asymptotic behavior of eigenvalues $\lambda_n$ with the
logarithmic decay rate for small deviations, we use the following statement, that is
slightly reformulated \cite[Theorem 2]{Na1}.
\begin{prop} 
 Let $(\lambda_n)$, $n\in\mathbb N$, be a
positive sequence with counting function ${\cal N}(t)$.
Suppose that
$${\cal N}(t)\sim \varphi(1/t), \qquad t\to 0+,
$$
where $\varphi$ is a function slowly varying at infinity.
Then, as $r\to 0+$,
\begin{equation}\label{vero}
\ln{\bf P}\left\{\sum_{n=1}^\infty\lambda_n\xi_n^2\le r
\right\} \sim -\frac 12 \int\limits_1^u \varphi(z)\,\frac {dz}z, 
\end{equation}
where $u=u(r)$ satisfies the relation
\begin{equation}\label{ur}
\frac {\varphi(u)}{2u}\sim r, \qquad r\to 0+.
\end{equation}
\end{prop}

{\bf Example 3}. Let 
$\varphi(\tau)=\ln^\alpha(\tau)\cdot\Phi(\ln(\tau))$, where $\alpha\ge0$ while $\Phi$ is a SVF\footnote{If $\alpha=0$, 
then we require in addition that $\Phi$ is nondecreasing and unbounded.}. Then, as $\ep\to 0+$,
\begin{equation}\label{smalldev}
 \ln{\bf P}\left\{\sum_{n=1}^\infty\lambda_n\xi_n^2\le \ep^2\right\} \sim 
-\frac{2^\alpha}{\alpha+1}\,\varphi(1/\ep)\ln(1/\ep).
\end{equation}

\begin{proof}
Changing the variable $z=\frac u\sigma$ and using (\ref{conv_pq}), we observe that
$$\int\limits_1^u \varphi(z)\,\frac {dz}z=(\varphi\star\ln)(u)\sim
\frac 1{\alpha+1}\,\varphi(u)\ln(u),\qquad u\to+\infty.
$$
Next, direct calculation shows that $u=\frac {\varphi(1/r)}{2r}$ satisfies (\ref{ur}). Therefore, formula (\ref{vero}) gives, as $r\to 0+$,
$$\ln{\bf P}\left\{\sum_{n=1}^\infty\lambda_n\xi_n^2\le r\right\} \sim 
-\frac 1{2(\alpha+1)} \,\varphi(1/r)\ln(1/r).
$$
Replacing $r$ by $\ep^2$, we arrive at (\ref{smalldev}).
\end{proof}

{\bf Example 4}. Let 
$\varphi(\tau)=\exp(a\ln^{\frac 1p}(\tau))\,\ln^\alpha(\tau)\cdot\Phi(\ln(\tau))$, 
where $\alpha\ge0$, $p>1$ while $\Phi$ is a SVF. Then, as $\ep\to 0+$,
\begin{multline}\label{smalldev1}
 \ln{\bf P}\left\{\sum_{n=1}^\infty\lambda_n\xi_n^2\le \ep^2\right\} \sim \\
\sim-\frac {2^{\alpha-\frac 1p}p}a\,\exp\Big[2\ln(1/\ep)\cdot
\sum\limits_{k=1}^{[p']}c_k\big(\frac a{2^{\frac 1{p'}}}
\ln^{-\frac 1{p'}}(1/\ep)\big)^k\Big]\,\ln^{\alpha+\frac 1{p'}}(1/\ep)
\cdot\Phi(\ln (1/\ep)),
\end{multline}
where
\begin{equation}\label{coeff}
c_1=1;\quad\qquad c_k=\frac 1{k!}\,\prod\limits_{m=0}^{k-2}\Big(\frac kp-m\Big)\quad\mbox{for} \quad k\ge2.
\end{equation}
In particular, if $p>2$ then
$$\ln{\bf P}\left\{\sum_{n=1}^\infty\lambda_n\xi_n^2\le \ep^2\right\} \sim 
-\frac {2^{-\frac 1p}p}a\,\varphi(\ep^2)\ln^{\frac 1{p'}}(1/\ep).
$$

\begin{proof}
Changing the variable we obtain
$$\int\limits_1^u\varphi(z)\,\frac {dz}z=T^{\alpha +1}\cdot\Phi(T)
\int\limits_0^1 \exp(a T^{\frac 1p}\vartheta^{\frac 1p})\cdot\vartheta^\alpha\cdot\frac {\Phi(T\vartheta))}{\Phi(T)}\,d\vartheta,
$$
where $T=\ln(u)$.

The Laplace method and the Lebesgue dominated convergence theorem give us the asymptotics
$$\int\limits_1^u\varphi(z)\,\frac {dz}z\sim \frac pa
\exp(a\ln^{\frac 1p} (u))\,\ln^{\alpha+\frac 1{p'}}(u)\cdot \Phi(\ln (u)),\qquad u\to+\infty$$
(we recall that $p'$ is the H\"older conjugate exponent for $p$).

Next, direct though cumbersome calculation shows that 
$$u=\frac 1{2r}\,\exp\Big[\ln(1/r)\cdot
\sum\limits_{k=1}^{[p']}c_k\big(a\ln^{-\frac 1{p'}}(1/r)\big)^k\Big]\,\ln^{\alpha}(1/r)\cdot
\Phi(\ln (1/r))$$
with $c_k$ given by (\ref{coeff}) satisfies (\ref{ur}). Therefore, formula (\ref{vero}) gives, 
as $r\to 0+$,
\begin{multline*}
\ln{\bf P}\left\{\sum_{n=1}^\infty\lambda_n\xi_n^2\le r\right\}
\sim-\frac pa\,ur\ln^{\frac 1{p'}}(u) \sim \\
\sim-\frac p{2a}\,\exp\Big[\ln(1/r)\cdot
\sum\limits_{k=1}^{[p']}c_k\big(a\ln^{-\frac 1{p'}}(1/r)\big)^k\Big]\,\ln^{\alpha+\frac 1{p'}}(1/r)
\cdot\Phi(\ln (1/r)).
\end{multline*}
Replacing $r$ by $\ep^2$, we arrive at (\ref{smalldev1}).
\end{proof}

Turning to specific fields, we start with a stationary sheet 
$${\cal R}^{\cal G}_d(x)=\underset{j=1}{\overset{d}\otimes} R^{{\cal G}_j}(x_j),\qquad
x=(x_1,\dots,x_d) \in [0,1]^d,
$$
where $R^{{\cal G}_j}$ are stationary Gaussian processes with zero mean-values and 
the spectral densities
$$h_{R^{{\cal G}_j}}(\xi)=\exp(-{\cal G}_j(\xi)), \qquad \xi\in\mathbb R$$
(here ${\cal G}_j$ is even and ${\cal G}_j(\xi)\to+\infty$ as $\xi\to+\infty$).

Assume for simplicity that ${\cal G}$ is smooth and $\xi{\cal G}'(\xi)\to+\infty$ as $\xi\to+\infty$. 
Then it is easy to check that the corresponding covariances $G_{R^{\cal G}}$ are smooth functions. 
For instance, it is well known that
$${\cal G}(\xi)=|\xi|\quad\Longrightarrow\quad
G_{R^{\cal G}}(s,t)=\frac 1{\pi(1+(s-t)^2)};
$$
$${\cal G}(\xi)=\xi^2\quad\Longrightarrow\quad
G_{{\cal G}}(s,t)=\frac 1{2\sqrt{\pi}}\exp\left(-\frac {(s-t)^2}{4}\right).
$$
Small deviations of such processes in various $L_p$-norms in the case ${\cal G}(\xi)=|\xi|^\alpha$ were 
considered in \cite{Na1}, \cite{AILZ}, \cite{Ku}.

\begin{prop}\label{exp_sheet}
{\bf 1}. Let ${\cal G}_j(\xi)\sim C_j\ln^p(\xi)$ as $\xi\to+\infty$, with $p>1$. Then, as 
$\ep\to 0+$,
\begin{multline}\label{log}
\ln {\bf P}\{\|{\cal R}^{\cal G}_d\|\le\ep\}\sim -\frac p2
\Big(\pi^{d+1}{\mathfrak B}{\mathfrak C}\Big(\frac {p-1}2\Big)^{d-1}
\Big)^{-\frac 12}\times\\
\times\exp\Big[2\ln(1/\ep)\cdot
\sum\limits_{k=1}^{[p']}c_k\Big(\frac {2\ln(1/\ep)}{\mathfrak B}\Big)^{-\frac k{p'}}\Big]
\Big(\frac {2\ln(1/\ep)}{\mathfrak B}\Big)^{1+\frac {d-3}{2p}},
\end{multline}
where
$${\mathfrak B}=\sum\limits_{j=1}^dC_j^{-\frac 1{p-1}};\qquad
{\mathfrak C}=\prod\limits_{j=1}^dC_j^{\frac 1{p-1}},
$$
while $c_k$ are given by (\ref{coeff}). In particular, if $p>2$ then
$$\ln {\bf P}\{\|{\cal R}^{\cal G}_d\|\le\ep\}\sim -\frac p2
\Big(\pi^{d+1}{\mathfrak B}{\mathfrak C}\Big(\frac {p-1}2\Big)^{d-1}
\Big)^{-\frac 12}\exp\Big[{\mathfrak B}^{\frac 1{p'}}\Big(2\ln(1/\ep)\Big)^{\frac 1p}\Big]
\Big(\frac {2\ln(1/\ep)}{\mathfrak B}\Big)^{1+\frac {d-3}{2p}}.
$$

{\bf 2}. Let ${\cal G}_j(\xi)\sim\xi^q\Phi_j(\xi)$ as $\xi\to+\infty$, with $0<q\le1$ and 
$\Phi_j$ being an SVF (if $q=1$ we require in addition that $\Phi_j(\xi)\to0$ as $\xi\to+\infty$).
Then, as $\ep\to 0+$,
\begin{equation}\label{q<1}
\ln {\bf P}\{\|{\cal R}^{\cal G}_d\|\le\ep\}\sim 
-\frac {2^{\frac dq}\Gamma^d(\frac {q+1}q)}{\pi^d\Gamma(\frac dq+2)}\cdot
\ln^{\frac dq+1}(1/\ep)\cdot\prod\limits_{j=1}^d\phi_j(\ln(1/\ep)),
\end{equation}
where $\phi_j$, $j=1,\dots,d$, are SVFs depending only on $\Phi_j$ and on $q$; for $q=1$ we have $\phi_j(t)\to+\infty$ as $t\to+\infty$. In particular,
$$\Phi_j(\xi)=C_j\ln^p(\xi)\quad\Longrightarrow\quad
\ln {\bf P}\{\|{\cal R}^{\cal G}_d\|\le\ep\}\sim 
-\frac {\Gamma^d(\frac {q+1}q)}{\pi^d\Gamma(\frac dq+2)}\cdot
\Big(\frac {2q^p}{{\mathfrak C}}\Big)^{\frac dq}\cdot
\frac {\ln^{\frac dq+1}(1/\ep)}{\ln^{\frac {p d}q}(\ln(1/\ep))},
$$
where ${\mathfrak C}=\Big(\prod\limits_{j=1}^d C_j\Big)^{\frac 1d}$ (we recall that $p<0$ for 
$q=1$).\medskip

{\bf 3}. Let ${\cal G}_j(\xi)\sim C_j\xi$ as $\xi\to+\infty$. Then, as $\ep\to 0+$,
\begin{equation}\label{q=1}
\ln {\bf P}\{\|{\cal R}^{\cal G}_d\|\le\ep\}\sim 
-\frac {2^d}{\pi^d(d+1)!\,{\mathfrak C}}\cdot\ln^{d+1}(1/\ep),
\end{equation}
where
$${\mathfrak C}=\prod\limits_{j=1}^d
\frac {{\bf K}(\mbox{\rm sech}(\pi/2C_j))}{{\bf K}(\tanh(\pi/2C_j))}$$
while $\bf K$ is the complete elliptic integral of the first kind, see, e.g., \cite[8.11]{GR}.\medskip

{\bf 4}. Let ${\cal G}_j(\xi)\sim\xi\Phi_j(\xi)$ as $\xi\to+\infty$, where $\Phi_j$ is an SVF.
Suppose in addition that ${\cal G}_j(\xi)$ is convex for large $\xi$ and $\Phi_j(\xi)\to+\infty$ 
as $\xi\to+\infty$. Then, as $\ep\to 0+$,
\begin{equation}\label{q=1>}
\ln {\bf P}\{\|{\cal R}^{\cal G}_d\|\le\ep\}\sim 
-\frac 1{(d+1)!}\cdot\frac {\ln^{d+1}(1/\ep)}{\prod\limits_{j=1}^d\ln(\phi_j(\ln(1/\ep)))},
\end{equation}
where $\phi_j$, $j=1,\dots,d$ are SVFs depending only on $\Phi_j$, $\phi_j(t)\le\Phi_j(t)$ and
$\phi_j(t)\to+\infty$ as $t\to+\infty$. In particular,
$$\Phi_j(\xi)=C_j\ln^p(\xi)\quad\Longrightarrow\quad
\ln {\bf P}\{\|{\cal R}^{\cal G}_d\|\le\ep\}\sim 
-\frac 1{p^d(d+1)!}\cdot
\frac {\ln^{d+1}(1/\ep)}{\ln^d(\ln(\ln(1/\ep)))}
$$
(we recall that $p>0$).\medskip

{\bf 5}. Let $\ln({\cal G}_j(\xi))\sim q\ln(\xi)$ as $\xi\to+\infty$, with $q>1$. Then, 
as $\ep\to 0+$,
\begin{equation}\label{q>1}
\ln {\bf P}\{\|{\cal R}^{\cal G}_d\|\le\ep\}\sim 
-\frac {q^d}{(q-1)^d(d+1)!}\cdot\frac {\ln^{d+1}(1/\ep)}{\ln^d(\ln(1/\ep))}.
\end{equation}

{\bf 6}. Let $\ln({\cal G}_j(\xi))/\ln(\xi)\to+\infty$ as $\xi\to+\infty$. Then, 
as $\ep\to 0+$,
\begin{equation}\label{q=infty}
\ln {\bf P}\{\|{\cal R}^{\cal G}_d\|\le\ep\}\sim 
-\frac 1{(d+1)!}\cdot\frac {\ln^{d+1}(1/\ep)}{\ln^d(\ln(1/\ep))}.
\end{equation}

\end{prop}

\begin{rem}
General formulas (\ref{log})--(\ref{q=infty}) are new even for $d=1$. A particular case of purely power
dependence of ${\cal G}_j$ with $d=1$ was considered in \cite{Na1}. The recent preprint \cite{Ku} deals 
with ${\cal G}_j(\xi)=C_j\xi^2$ for arbitrary $d$ but does not contain exact constant in (\ref{q>1}).

Note that in the superexponential case (parts {\bf 4}--{\bf 6}) the logarithmic small ball asymptotics 
does not change if one multiplies ${\cal G}_j$ by a constant.
\end{rem}

\begin{proof}
{\bf 1}. It is shown in the remarkable paper \cite{W} that, if ${\cal G}_j(\xi)/\xi\to0$ as
$\xi\to+\infty$, then
\begin{equation}\label{Widom1}
\lambda_n^{(j)}\sim\exp(-{\cal G}_j(\pi n\cdot(1+o(1))))\qquad \mbox{as}\quad n\to\infty.
\end{equation}
In our case (\ref{Widom1}) implies
$${\cal N}_j(t)\sim \pi^{-1} \cdot\exp(C_j^{-\frac 1p}\ln^{\frac 1p}(1/t)),\qquad t\to 0+.
$$
Using Theorem \ref{spectr_asymp} and Example 2 successively $d-1$ times, we obtain that 
$${\cal N}(t)\sim 
\Big(\pi^{d+1}{\mathfrak C}{\mathfrak B}^{1+\frac {d-1}p}\Big(\frac {p-1}2\Big)^{d-1}
\Big)^{-\frac 12}
\exp({\mathfrak B}^{\frac 1{p'}}\ln^{\frac 1p}(1/t))\ln^{\frac {d-1}{2p}}(1/t).
$$
By (\ref{smalldev1}) we arrive at (\ref{log}).\medskip

{\bf 2}. It follows from (\ref{Widom1}) that in this case
$${\cal N}_j(t)\sim \pi^{-1} \cdot\ln^{\frac 1q}(1/t)\phi_j(\ln(1/t)),\qquad t\to 0+,
$$
and in a mentioned particular case
$${\cal N}_j(t)\sim \pi^{-1}\ \bigg(\frac {q^p}{C_j}
\cdot\frac{\ln(1/t)}{\ln^p(\ln(1/t))}\bigg)^{\frac 1q},\qquad t\to 0+.
$$
Using Theorem \ref{spectr_asymp} and Example 1 successively $d-1$ times, we obtain that 
$${\cal N}(t)\sim \frac {\Gamma^d(\frac {q+1}q)}
{\pi^d\Gamma(\frac dq+1)\,}\cdot
\ln^{\frac dq}(1/t)\cdot\prod\limits_{j=1}^d\phi_j(\ln(1/t)).
$$
By (\ref{smalldev}) we arrive at (\ref{q<1}).\medskip

{\bf 3}. It follows from \cite{W}, that in this case
$${\cal N}_j(t)\sim\frac {{\bf K}(\tanh(\pi/2C_j))}{\pi\, {\bf K}(\mbox{sech}(\pi/2C_j))}\cdot\ln(1/t),
\qquad t\to 0+.
$$
Similarly to the case {\bf 2}, we arrive at (\ref{q=1}).\medskip

{\bf 4}. It is proved in \cite{W}, that, if ${\cal G}_j(\xi)$ is convex for large $\xi$ and
${\cal G}_j(\xi)/\xi\to+\infty$ as $\xi\to+\infty$, then
\begin{equation}\label{Widom2}
\ln(\lambda_n^{(j)})\sim-{\cal G}_j(\xi(n))\qquad \mbox{as}\quad n\to\infty,
\end{equation}
where $\xi(n)<n$ is a unique solution of the equation
$${\cal G}_j(\xi)=2n\ln\Big(\frac n{\xi}\Big).
$$

In this case we obtain from (\ref{Widom2})
$${\cal N}_j(t)\sim\frac {\ln(1/t)}{2\ln(\phi_j(\ln(1/t)))},\qquad t\to 0+.
$$
and in a mentioned particular case
$${\cal N}_j(t)\sim \frac 1{2p}\cdot\frac {\ln(1/t)}{\ln(\ln(\ln(1/t)))},\qquad t\to 0+.
$$
Similarly to the case {\bf 2}, we arrive at (\ref{q=1>}).

{\bf 5}. It follows from (\ref{Widom2}), that in this case
$${\cal N}_j(t)\sim\frac 1{2-\frac 2q} \cdot\frac {\ln(1/t)}{\ln(\ln(1/t))},\qquad t\to 0+.
$$
Similarly to the case {\bf 2}, we arrive at (\ref{q>1}).

{\bf 6}. It follows from (\ref{Widom2}), that in this case
$${\cal N}_j(t)\sim\frac 12 \cdot\frac {\ln(1/t)}{\ln(\ln(1/t))},\qquad t\to 0+.
$$
Similarly to the case {\bf 2}, we arrive at (\ref{q=infty}).
\end{proof}

Now we consider a smooth homogeneous sheet 
$${\cal Z}^{({\mathfrak a},{\mathfrak b})}_d(x)=
\underset{j=1}{\overset{d}\otimes} Z^{({\mathfrak a}_j,{\mathfrak b}_j)}(x_j),\qquad
x\in [0,1]^d,
$$
where $Z^{({\mathfrak a}_j,{\mathfrak b}_j)}$ are Gaussian processes with zero mean-values and 
the covariances 
$$G_{Z^{({\mathfrak a}_j,{\mathfrak b}_j)}}(s,t)=
\frac{s^{{\mathfrak b}_j}t^{{\mathfrak b}_j}}{(s+t)^{{\mathfrak a}_j}}.
$$
Some properties of these processes were studied in \cite{LiS2} (for 
${\mathfrak b}=\frac 12 ({\mathfrak a}+1)$) and \cite{AGKLS}.

\begin{prop}\label{hom_sheet}
Let ${\mathfrak a}_j>0$, and 
${\mathfrak c}_j\equiv 2{\mathfrak b}_j-{\mathfrak a}_j+1>0$, $j=1,\dots,d$.
Then, as $\ep\to 0+$,
\begin{equation}\label{laptev}
\ln {\bf P}\{\|{\cal Z}^{({\mathfrak a},{\mathfrak b})}_d\|\le\ep\}\sim 
-\frac {2^{2d}}{\pi^{2d}(2d+1)!\,{\mathfrak C}}\cdot\ln^{2d+1}(1/\ep),
\end{equation}
where ${\mathfrak C}=\prod\limits_{j=1}^d {\mathfrak c}_j$.
\end{prop}

\begin{proof}
 It is shown in \cite{La}, that 
$${\cal N}_j(t)\sim \big(2\pi^2 {\mathfrak c}_j\big)^{-1} \cdot\ln^2(1/t),
\qquad t\to 0+.
$$
Similarly to the Proposition \ref{exp_sheet}, case {\bf 2}, we arrive at (\ref{laptev}).
\end{proof}

The next example deals with conventional Brownian sheet
\begin{equation}\label{sheet}
{\mathbb W}_d(x)=\underset{j=1}{\overset{d}\otimes} W_j(x_j),\qquad x\in [0,1]^d,
\end{equation}
where $W_j$ are Wiener processes.

As mentioned in the Introduction, the logarithmic asymptotics of small deviations 
for the Brownian sheet in $L_2$-norm with respect to the Lebesgue measure was obtained 
in \cite{Cs}. In \cite{KNN}, as a particular case of Proposition 5.1
(see also Theorem 8.2), this result was generalized to the case of absolutely continuous
measure with arbitrary continuous density. Now we are going to obtain the logarithmic
$L_2$-small ball asymptotics in the case of the discrete, degenerately self-similar measure
(see \cite{VSh}, \cite{NSh}). We recall briefly the construction of this measure (for $d=1$).\medskip

Let $0=\alpha_1<\alpha_2<\ldots<\alpha_M<\alpha_{M+1}=1$, $M\ge2$, be a partition
of the segment $[0,1]$. For some $m\in\{1,\dots,M\}$, denote by $a=\alpha_{m+1}-\alpha_m$
the length of $[\alpha_m,\alpha_{m+1}]$. Also we introduce a real number $\delta$ and
a real vector $(\beta_k)$, $k=1,\ldots,M$, such that
\begin{enumerate}
\item $0<\delta<1$;
\item $\beta_1=0$;\qquad $\chi_{\{m=M\}}\cdot\delta+\beta_M=1$;
\item $\beta_k<\beta_{k+1}$,\quad $k=1,\dots,M-1$;\qquad $\delta\beta_M+\beta_m<\beta_{m+1}$
\end{enumerate}
(for $m=M$ the last inequality is irrelevant).

It is shown in \cite[\S 2]{NSh} that under these conditions there exists a unique function $f$
such that ${\cal S}(f)=f$, where ${\cal S}$ is the simplest {\it similarity operator}
\begin{equation*}\label{eq:auxto}
{\cal S}[f](t)=\delta\cdot f(a^{-1}(t-\alpha_m))+
\sum\limits_{k=1}^M\beta_k\cdot\chi_{\,]\alpha_k,\alpha_{k+1}[}(t).
\end{equation*}
Moreover, $f$ increases on $[0,1]$, and $f(0)=0$, $f(1)=1$. The derivative $f'$ in the sence of 
distributions is a discrete probability measure $\mu$ on $[0,1]$ with
a unique singular point $\widehat x=\frac{\alpha_m}{1-a}$. It is called {\it degenerately self-similar 
measure}, generated by the set of parameters $M$, $m$, $\delta$, $(\alpha_k)$ and $(\beta_k)$.

\begin{prop}\label{brownian_sheet}
Let the measure $\mu$ on $[0,1]^d$ be the tensor product:
$$\mu(dx)=\underset{j=1}{\overset{d}\otimes}\mu_j(dx_j),
$$ 
where $\mu_j$, $j=1,\dots,d$, are degenerately self-similar measures, generated, respectively, by the 
sets of parameters
$$ M_j,\quad m_j,\quad \delta_j,\quad (\alpha_k^{(j)}),\quad (\beta_k^{(j)}),\qquad j=1,\dots,d.
$$
Then, as $\ep\to 0+$,
\begin{equation}\label{samopod}
\ln {\bf P}\{\|{\mathbb W}_d\|_{\mu}\le\ep\}\sim 
-\frac {2^d\,{\mathfrak C}}{(d+1)!}\cdot\ln^{d+1}(1/\ep),
\end{equation}
where
$${\mathfrak C}=\prod\limits_{j=1}^d\frac {M_j-1}{\ln(a_j\delta_j)}
$$
(we recall that $a_j=\alpha^{(j)}_{m_j+1}-\alpha^{(j)}_{m_j}$).
\end{prop}

\begin{proof}
It is shown in \cite{NSh}, that
$${\cal N}_j(t)\sim \frac {M_j-1}{\ln(a_j\delta_j)}\cdot\ln(1/t), 
\qquad t\to 0+.
$$
Similarly to the Proposition \ref{exp_sheet}, case {\bf 2}, we arrive at (\ref{samopod}).
\end{proof}

\begin{rem}
Note, analogously to Remark 6 in \cite{KNN}, that the replacement of any factor in (\ref{sheet}) by the
Brownian bridge, by the Ornstein--Uhlenbeck process or by similar process does not influence on the relation 
(\ref{samopod}). For instance, the Brownian pillow ${\mathbb B}_d(x)=\underset{j=1}{\overset{d}\otimes}\ B(x_j)$ satisfies, as $\ep\to0$, the relation 
$$\ln {\bf P}\{\|{\mathbb B}_d\|_{\mu}\le\ep\} \sim \ln {\bf P}\{\|{\mathbb W}_d\|_{\mu}\le\ep\}.$$
\end{rem}

Now we consider the {\it isotropically integrated} Brownian sheet 
$$({\mathbb W}_d)_{\mathfrak s}(x)=\underset{j=1}{\overset{d}\otimes} W_{\mathfrak s}(x_j),$$
where
$$W_{\mathfrak s}(t)\equiv W_{\mathfrak s}^{[b_1,\,\dots,\,b_{\mathfrak s}]}(t) =
(-1)^{b_1+\,\dots\,+b_{\mathfrak s}}\underbrace
{\int\limits_{b_{\mathfrak s}}^t\dots\int\limits_{b_1}^ {t_1}}_ {{\mathfrak s}} \
W (s)\ ds\ dt_1\dots$$
(here any $b_k$ equals either zero or one, $t\in[0,1]$; for various $j=1,\dots,d$ the multi-indices
$[b_1,\,\dots,\,b_{\mathfrak s}]$ can differ). 

\begin{prop}
Let a measure $\mu$ on $[0,1]^d$ be as in Proposition $\ref{brownian_sheet}$. 
Then, as $\ep\to 0+$,
\begin{equation}\label{samopod1}
\ln {\bf P}\{\|({\mathbb W}_d)_{\mathfrak s}\|_{\mu}\le\ep\}\sim 
-\frac {2^d\,{\mathfrak C}}{(d+1)!}\cdot\ln^{d+1}(1/\ep),
\end{equation}
where
$${\mathfrak C}=\prod\limits_{j=1}^d\frac {M_j-1}{\ln(a_j\delta_j^{2{\mathfrak s}+1})}.
$$
\end{prop}

\begin{proof}
This statement can be proved in the same way as Proposition \ref{brownian_sheet}.
\end{proof}

Finally, we can consider the fields-products corresponding to
essentially different marginal processes. We restrict ourselves
to a single example. On $[0,1]^2$ consider the Gaussian field
${\mathfrak R}(x_1)\otimes Z^{({\mathfrak a}, {\mathfrak b})}(x_2)$,
where ${\mathfrak R}$ is a stationary Gaussian process with zero mean-value and 
the spectral density $h_{\mathfrak R}(\xi)=\frac 1{\Gamma(|\xi|)}$.

\begin{prop} 
Let ${\mathfrak a}>0$, and ${\mathfrak c}\equiv 2{\mathfrak b}-{\mathfrak a}+1>0$. Then
$$\ln {\bf P}\{\|{\mathfrak R}\otimes Z^{({\mathfrak a}, {\mathfrak b})}\|\le\ep\}\sim
-\frac 1{6\pi^2{\mathfrak c}}\cdot\frac {\ln^4(1/\ep)}{\ln(\ln(\ln(1/\ep)))}.
$$
\end{prop}

\begin{proof}
The asymptotics of marginal counting functions are calculated in Proposition \ref{exp_sheet}, 
part {\bf 4}, and in Proposition \ref{hom_sheet}. The result follows from Example 1 and (\ref{smalldev}).
\end{proof}

\bigskip
We are grateful to M.A.~Lifshits for useful comments and to Ya.Yu.~Nikitin for the constant encouragement.
We also thank T.~K\"uhn who kindly provided us with the text of preprint \cite{Ku}.

\end{document}